\documentclass[12pt]{article}
\usepackage{amsmath,amssymb,amsthm, comment}

\newtheorem{theorem}{Theorem}[section]
\newtheorem{lemma}[theorem]{Lemma}

\newtheorem{example}[theorem]{Example}

\def\barr{\begin{array}}
\def\earr{\end{array}}
\title{Isolated subgroups of finite abelian groups}
\author{Marius T\u arn\u auceanu}
\date{February 9, 2021}

\begin{document}

\maketitle

\begin{abstract}
We say that a subgroup $H$ is isolated in a group $G$ if for every $x\in G$ we have either $x\in H$ or $\langle x\rangle\cap H=1$. In this short note, we describe the set of isolated subgroups of a finite abelian group. The technique used is based on an interesting connection between isolated subgroups and the function sum of element orders of a finite group.
\end{abstract}

{\small
\noindent
{\bf MSC2020\,:} Primary 20K01; Secondary 20K27.

\noindent
{\bf Key words\,:} finite abelian groups, isolated subgroups, sum of element orders.}
\vspace{-2mm}

\section{Introduction}

Let $G$ be a finite group. We say that a subgroup $H$ of $G$ is isolated in $G$ if for every $x\in G$ we have either $x\in H$ or $\langle x\rangle\cap H=1$.
The starting point for our discussion is given by the Z. Janko's paper \cite{3} that investigates isolated subgroups for certain classes of nonabelian $p$-groups. In the following, we will determine these subgroups for finite abelian groups. The main tool used is the function sum of element orders of $G$,
\begin{equation}
\psi(G)=\sum_{x\in G}o(x),\nonumber
\end{equation}defined by H. Amiri, S.M. Jafarian Amiri and I.M. Isaacs in \cite{1}. Given a subgroup $H$ of $G$, this has been generalized in \cite{6} to the function
\begin{equation}
\psi_H(G)=\sum_{x\in G}o_H(x),\nonumber
\end{equation}where $o_H(x)$ denotes the order of $x$ relative to $H$, i.e. the smallest positive integer $m$ such that $x^m\in H$. Clearly, for $H=1$ we have $\psi_H(G)=\psi(G)$.\newpage

We remark that

\begin{equation}
\psi_H(G)=\sum_{x\in H}o_H(x)+\sum_{x\in G\setminus H}o_H(x)=|H|+\sum_{x\in G\setminus H}\frac{o(x)}{|\langle x\rangle\cap H|}\nonumber
\end{equation}and therefore $H$ is isolated in $G$ if and only if

\begin{equation}
\psi_H(G)=|H|+\sum_{x\in G\setminus H}o(x)=|H|+\psi(G)-\psi(H).\nonumber
\end{equation}Since for $H\lhd G$ we have $\psi_H(G)=|H|\psi(G/H)$, we infer that a normal subgroup $H$ is isolated in $G$ if and only if

\begin{equation}
\psi(G)-\psi(H)=|H|\left(\psi(G/H)-1\right).
\end{equation}

In particular, this equivalence holds for all subgroups $H$ of a finite abelian group $G$. It will be used in what follows, together with Theorem 1 of \cite{5}:

\begin{theorem}
Let $G=\mathbb{Z}_{p^{\alpha_1}}\times\mathbb{Z}_{p^{\alpha_2}}\times\cdots\times\mathbb{Z}_{p^{\alpha_k}}$ be a finite abelian
$p$-group, where $1\leq\alpha_1\leq\alpha_2\leq...\leq\alpha_k$. Then
\begin{equation}
\psi(G)=1+\sum_{\alpha=1}^{\alpha_k}\left(p^{2\alpha}f_{(\alpha_1,\alpha_2,...,\alpha_k)}(\alpha)-p^{2\alpha-1}f_{(\alpha_1,\alpha_2,...,\alpha_k)}(\alpha-1)\right),
\end{equation}where
\begin{equation}
f_{(\alpha_1,\alpha_2,...,\alpha_k)}(\alpha)=\left\{\barr{lll}
p^{(k-1)\alpha},&\mbox{ if }&0\le\alpha\le\alpha_1\vspace*{1,5mm}\\
p^{(k-2)\alpha+\alpha_1},&\mbox{ if }&\alpha_1\le\alpha\le\alpha_2\\
\vdots\\
p^{\alpha_1+\alpha_2+...+\alpha_{k-1}},&\mbox{ if
}&\hspace{-1mm}\alpha_{k-1}\le\alpha\,.\earr\right.\nonumber
\end{equation}
\end{theorem}

We note that (2) gives a formula for the sum of element orders of an arbitrary finite abelian group because the function $\psi$ is multiplicative. Also, we note that $\psi(G)$ in Theorem 1.1 is a polynomial in $p$ of degree $2\alpha_k+\alpha_{k-1}+...+\alpha_1$. An alternative way of writing it is
\begin{equation}
\barr{lcl}
\psi(G)&=&p^{2\alpha_k+\alpha_{k-1}+...+\alpha_1}-\left(p-1\right)\displaystyle\sum_{\alpha=0}^{\alpha_k-1}p^{2\alpha}f_{(\alpha_1,\alpha_2,...,\alpha_k)}(\alpha)\vspace{1,5mm}\\
&=&p^{2\alpha_k+\alpha_{k-1}+...+\alpha_1}+...+p^{k+1}-p+1.\earr\nonumber
\end{equation}

Most of our notation is standard and will usually not be repeated here. Elementary notions and results on groups can be found in \cite{2,4}.

\section{Main results}

We start with the following lemma whose proof is elementary and will be omitted.

\begin{lemma}
Let $G=G_1\times G_2\times\cdots\times G_m$ be a finite abelian group, where $\gcd(|G_i|,|G_j|)=1$ for all $i\neq j$, and $H=H_1\times H_2\times\cdots\times H_m$ be a subgroup of $G$. Then $H$ is isolated in $G$ if and only if there are $i_1,i_2,...,i_k\in\{1,2,...,m\}$ such that $H_{i_j}$ is isolated in $G_{i_j}$ for all $j=1,2,...,k$ and $H_i=1$ for all $i\neq i_1,i_2,...,i_k$.
\end{lemma}

Lemma 2.1 shows that our study can be reduced to finite abelian $p$-groups.

\begin{lemma}
Let $G=\mathbb{Z}_{p^{\alpha_1}}\times\mathbb{Z}_{p^{\alpha_2}}\times\cdots\times\mathbb{Z}_{p^{\alpha_k}}$ be a finite abelian
$p$-group, where $1\leq\alpha_1\leq\alpha_2\leq...\leq\alpha_k$, and $H$ be a maximal subgroup of $G$. If $H$ is isolated in $G$, then $G$ is elementary abelian and $H$ is a direct factor of $G$.
\end{lemma}

\begin{proof}
Let $n=\alpha_1+\alpha_2+...+\alpha_k$. Since $H$ is maximal and isolated in $G$, by the equality (1) it follows that
\begin{equation}
\psi(G)-\psi(H)=p^{n-1}\left(\psi(\mathbb{Z}_p)-1\right)=p^{n+1}-p^n\nonumber
\end{equation}and so $\psi(G)$ is a polynomial in $p$ of degree $n+1$. Thus $n+\alpha_k=n+1$, that is $\alpha_k=1$, implying that $G$ is elementary abelian. The second conclusion is obvious.
\end{proof}

From Lemma 2.2 we infer that if $H$ is an isolated subgroup of a finite abelian $p$-group $G$ and $H\neq G$, then
\begin{equation}
H\subset\Omega_1(G)=\{x\in G\mid x^p=1\}.
\end{equation}Indeed, take a subgroup $K$ of $G$ such that $H$ is maximal in $K$. Then $H$ is isolated in $K$ and Lemma 2.2 shows that $K$ must be elementary abelian, i.e. $K\subseteq\Omega_1(G)$. Hence $H$ is strictly contained in $\Omega_1(G)$.

\begin{lemma}
Let $G=\mathbb{Z}_{p^{\alpha_1}}\times\mathbb{Z}_{p^{\alpha_2}}\times\cdots\times\mathbb{Z}_{p^{\alpha_k}}$ be a finite abelian
$p$-group with $1<\alpha_1\leq\alpha_2\leq...\leq\alpha_k$. Then $G$ has no isolated proper subgroup.
\end{lemma}

\begin{proof}
Assume that $H$ is an isolated proper subgroup of $G$ and let $|H|=p^m$. Then
\begin{equation}
\psi(G)-\psi(H)=p^m\left(\psi(G/H)-1\right).
\end{equation}

By (3) we know that $H$ is elementary abelian and $m<k$. Then $\psi(H)=p^{m+1}-p+1$ and therefore the left side of (4) becomes
\begin{equation}
\psi(G)-\psi(H)=p^{2\alpha_k+\alpha_{k-1}+...+\alpha_1}+...+p^{k+1}-p^{m+1}.\nonumber
\end{equation}

On the other hand, (3) shows that $G/H$ has also $k$ direct factors
\begin{equation}
G/H=\mathbb{Z}_{p^{\beta_1}}\times\mathbb{Z}_{p^{\beta_2}}\times\cdots\times\mathbb{Z}_{p^{\beta_k}},\nonumber
\end{equation}where either $\beta_i=\alpha_i$ or $\beta_i=\alpha_i-1$, for all $i=1,2,...,k$. Thus the right side of (4) becomes
\begin{equation}
\barr{lcl}
p^m\left(\psi(G/H)-1\right)&=&p^m\left(p^{2\beta_k+\beta_{k-1}+...+\beta_1}+...+p^{k+1}-p\right)\vspace{1,5mm}\\
&=&p^{m+2\beta_k+\beta_{k-1}+...+\beta_1}+...+p^{m+k+1}-p^{m+1}.\earr\nonumber
\end{equation}

Hence (4) leads to $p^{k+1}=p^{m+k+1}$, i.e. $m=0$, a contradiction.
\end{proof}

In the following, let $G=\mathbb{Z}_{p^{\alpha_1}}\times\mathbb{Z}_{p^{\alpha_2}}\times\cdots\times\mathbb{Z}_{p^{\alpha_k}}$ be a finite abelian
$p$-group, where $1=\alpha_1=\alpha_2=...=\alpha_r<\alpha_{r+1}\leq...\leq\alpha_k$, and $H$ be a subgroup of order $p$ of $G$. Then $H$ is isolated in $G$ if and only if $\langle x\rangle\cap H=1$, $\forall\, x\in G\setminus H$, that is if and only if $H$ is contained in no cyclic subgroup of order $p^s$ with $s\geq 2$ of $G$. This is equivalent with the fact that $H$ is not a subgroup of order $p$ of $\mathbb{Z}_{p^{\alpha_{r+1}}}\times\cdots\times\mathbb{Z}_{p^{\alpha_k}}$. It is well-known that there are $q=\frac{p^{k-r}-1}{p-1}$ such subgroups, say $H_i$, $i=1,2,...,q$. Consequently,
$$H \mbox{ is isolated in } G \mbox{ if and only if } H\neq H_i, \forall\, i=1,2,...,q.$$Moreover, we infer that $G$ possesses
\begin{equation}
\frac{p^k-1}{p-1}-\frac{p^{k-r}-1}{p-1}=p^{k-r}\frac{p^r-1}{p-1}\nonumber
\end{equation} isolated subgroups of order $p$.

The above remark can be easily extended to subgroups $H$ of arbitrary orders of $G$, namely
$$H \mbox{ is isolated in } G \mbox{ if and only if } H_i\nsubseteq H, \forall\, i=1,2,...,q.$$Hence we proved the next lemma.

\begin{lemma}
Let $G=\mathbb{Z}_{p^{\alpha_1}}\times\mathbb{Z}_{p^{\alpha_2}}\times\cdots\times\mathbb{Z}_{p^{\alpha_k}}$ be a finite abelian
$p$-group, where $1=\alpha_1=\alpha_2=...=\alpha_r<\alpha_{r+1}\leq...\leq\alpha_k$, and let $A=\mathbb{Z}_{p^{\alpha_{r+1}}}\times\cdots\times\mathbb{Z}_{p^{\alpha_k}}$. Then a subgroup $H$ of $G$ is isolated in $G$ if and only if $H\cap A=1$.
\end{lemma}

In particular, Lemma 2.4 shows that all subgroups of an elementary abelian $p$-group are isolated.

We summarize the results in Lemmas 2.3 and 2.4 as follows.\newpage

\begin{theorem}
Let $G=\mathbb{Z}_{p^{\alpha_1}}\times\mathbb{Z}_{p^{\alpha_2}}\times\cdots\times\mathbb{Z}_{p^{\alpha_k}}$ be a finite abelian
$p$-group, where $1\leq\alpha_1\leq\alpha_2\leq...\leq\alpha_k$.
\begin{itemize}
\item[{\rm a)}] If $\alpha_1>1$, then the unique isolated subgroups of $G$ are $1$ and $G$.
\item[{\rm b)}] If $1=\alpha_1=\alpha_2=...=\alpha_r<\alpha_{r+1}\leq...\leq\alpha_k$ and $A=\mathbb{Z}_{p^{\alpha_{r+1}}}\times\cdots\times\mathbb{Z}_{p^{\alpha_k}}$, then the isolated subgroups of $G$ are $G$ and all subgroups $H\leq G$ with $H\cap A=1$.
\end{itemize}
\end{theorem}

Finally, we mention that the computation of isolated subgroups of finite abelian $p$-groups can be made by using the well-known Goursat's lemma (see e.g. the result (4.19) of \cite{4}). We exemplify it in three particular cases:

\begin{example}
\begin{itemize}
\item[{\rm 1.}] The group $\mathbb{Z}_p\times\mathbb{Z}_{p^m}$ with $m\geq 2$ has $p+2$ isolated subgroups, namely $1$, $G$ and $p$ subgroups of order $p$.
\item[{\rm 2.}] The group $\mathbb{Z}_p\times\mathbb{Z}_{p^m}\times\mathbb{Z}_{p^n}$ with $m,n\geq 2$ has $p^2+2$ isolated subgroups, namely $1$, $G$ and $p^2$ subgroups of order $p$.
\item[{\rm 3.}] The group $\mathbb{Z}_p\times\mathbb{Z}_p\times\mathbb{Z}_{p^m}$ with $m\geq 2$ has $2p^2+p+2$ isolated subgroups, namely $1$, $G$, $p^2+p$ subgroups of order $p$ and $p^2$ subgroups of order $p^2$.
\end{itemize}
\end{example}

\vspace*{3ex}\small

\hfill
\begin{minipage}[t]{5cm}
Marius T\u arn\u auceanu \\
Faculty of  Mathematics \\
``Al.I. Cuza'' University \\
Ia\c si, Romania \\
e-mail: {\tt tarnauc@uaic.ro}
\end{minipage}

\end{document}